\documentclass[12pt]{amsart}

\usepackage{amsfonts,amsmath,amssymb,amsthm,graphicx,enumerate, mathabx}
\usepackage[english]{babel}
\usepackage[utf8]{inputenc}
\usepackage[T1]{fontenc}

\usepackage{framed}

\usepackage{marginnote,color}

\usepackage{array}
\usepackage{hyperref}

\usepackage[margin=2.8cm]{geometry}

\newcommand{\defeq}{\mathrel{\mathop:}=}

\newcommand{\R}{\mathbb{R}}

\newcommand{\E}{\mathbb{E}}

\renewcommand{\P}{\mathbb{P}}

\newcommand{\N}{\mathbb{N}}
\newcommand{\Z}{\mathbb{Z}}
\newcommand{\indic}{{\bf 1}}

\newcommand{\pb}{\overline p}

\newcommand{\Nt}{\widetilde N}

\DeclareMathOperator*{\limup}{\lim\!\!\uparrow}

\newcommand{\dm}{\begin{pmatrix}} 
\newcommand{\fm}{\end{pmatrix}}
\newcommand{\ddm}{\begin{vmatrix}} 
\newcommand{\fdm}{\end{vmatrix}}

\usepackage{textcomp,eurosym}

\newcommand{\st}{\,:\,} 

\newtheorem{theorem}{Theorem}
\newtheorem{lemma}{Lemma}
\newtheorem{proposition}{Proposition}

\newtheorem*{remark*}{Remark}

\author[V. Sidoravicius]{Vladas Sidoravicius}
\address{Courant Institute of Mathematical Sciences, New York\\
NYU-ECNU Institute of Mathematical Sciences at NYU Shanghai}
\email{vs1138@nyu.edu}

\author[L. Tournier]{Laurent Tournier}
\address{LAGA, Universit\'e Paris 13, Sorbonne Paris Cit\'e, CNRS, UMR 7539, 93430 Villetaneuse, France. 
 }
\email{tournier@math.univ-paris13.fr}

\title[]{The ballistic annihilation threshold is $1/4$}

\begin{document}

\ 
\vspace{-1cm}
\maketitle

{\footnotesize \noindent{\slshape\bfseries Abstract.} 
We consider a system of annihilating particles where particles start from the points of a Poisson process on the line, move at constant i.i.d.\ speeds symmetrically distributed in $\{-1,0,+1\}$ and annihilate upon collision. We prove that particles with speed 0 vanish almost surely if and only if their initial density is smaller than or equal to $1/4$, and give an explicit formula for the probability of survival of a stationary particle in the supercritical case, which is in accordance with the predictions of~\cite{droz1995ballistic}. The present proof relies essentially on an identity proved in J.\,Haslegrave's recent paper \cite{haslegrave}. }

\

\textit{Since a key ingredient in our proof comes from the very recent paper~\cite{haslegrave}, we borrow all notations from that paper in order to ease a joint reading, and refer to it for definitions. 
Let us also only give a brief description, and refer to either~\cite{haslegrave} or~\cite{sidoravicius-tournier} for a more elaborate introduction. }

We consider a system of particles, known as a particular case of \emph{ballistic annihilation}, whose initial condition is random but whose subsequent time evolution follows deterministically. Let $p\in[0,1]$. At time 0, particles (or ``bullets'') are located at the points of a Poisson point process of intensity 1 on the real line, and endowed with i.i.d.\ speeds distributed according to the symmetric law $\frac{1-p}2\delta_{-1}+p\delta_0+\frac{1-p}2\delta_1$. Then, all particles move along the line at their given constant speed, and pairs of particles annihilate as soon as they collide. We are interested in the eventual survival of particles with speed 0 (or ``stationary'' particles). Due to symmetries, it is in fact often more convenient to only consider particles on the real half-line $(0,\infty)$ and wonder whether 0 is eventually reached by a particle or not. For that reason, unless specified otherwise by a subscript $\R$, we henceforth consider that the model is defined on $(0,\infty)$. 

Our main result is an explicit computation of the probability that a stationary particle survives, which confirms predictions from the physics literature~\cite{droz1995ballistic}. A rigorous proof of the fact that stationary particles vanish for small $p$ had been a surprisingly difficult open problem until Haslegrave's recent breakthrough~\cite{haslegrave}. Although the algebraic nature of the proof in~\cite{haslegrave} still didn't provide much probabilistic insight, it opened the way to explicit computations and thus to a possible confirmation of some facts from~\cite{droz1995ballistic}. This program is achieved in the present paper by providing an algebraic identity that combines with the ones in~\cite{haslegrave} into a solvable equation, and by discussing the topological arguments required to identify the relevant solution of that equation and get the following full conclusion. 

\begin{theorem}\label{thm:main}
This model undergoes a phase transition at $p_c=\frac14$. More precisely, the probability that $0$ is reached by a particle on $(0,\infty)$ is, for all $p\in(0,1]$, 
\[q\defeq\P(0\leftarrow\bullet)=\begin{cases} 1 & \text{if $p\le1/4$}\\\frac1{\sqrt p}-1&\text{if $p>1/4$.}\end{cases}\]
Equivalently, the probability that a given stationary particle survives in the full line process is
\[\theta(p)\defeq\P((\bullet\not\rightarrow0)_\R\wedge(0\not\leftarrow\bullet)_\R)=\begin{cases} 0 & \text{if $p\le1/4$}\\ \Big(2-\frac1{\sqrt p}\Big)^2 & \text{if $p>1/4$.}\end{cases}\]
\end{theorem}

The proof decomposes into two parts. First, combining the key identity from~\cite{haslegrave} with a new identity, we derive algebraically that, for all $p$, the probability  $q$ is either equal to $1$ or to $\frac1{\sqrt p}-1$. This entails $q=1$ when $p\le1/4$, but doesn't prove the converse. For the latter, we need a priori regularity properties of $q$ (or $\theta$) as a function of $p$. Unfortunately, we cannot rely of monotonicity since the apparent lack thereof is precisely a major difficulty in this model. 

In a second part, we use finitary conditions characterizing the survival phase together with the previous dichotomy in order to show that $q<1$ on the whole interval $(\frac14,1)$, which completes the proof of our main theorem. 

A last section of the paper contains the definition and properties of one of the above mentioned finitary conditions. 

\section{Algebraic identities}

In this section, we prove
\begin{proposition}\label{prop:dichotomy}
For all $p\in(0,1)$, $q=1$ or $q=\frac1{\sqrt p}-1$. In particular, $q=1$ if $p\le\frac14$. 
\end{proposition}

Let us first recall a key identity from \cite{haslegrave}, relating $p$, $q$ and
\[r\defeq\P((\vec\bullet_1\rightarrow\dot\bullet)\wedge(0\leftarrow\bullet)).\]

\begin{lemma}[{\cite[Lemma 2]{haslegrave}}]\label{lemma1}
$q=\frac{1-p}2(1+q)+r(1-q)+pq^3$.
\end{lemma}

The conclusion will follow from the next lemma:

\begin{lemma}\label{lemma2}
$r=\frac12pq^2$.
\end{lemma}

\begin{proof}[{Proof of Proposition~\ref{prop:dichotomy}}]
Combining Lemmas~\ref{lemma1} and~\ref{lemma2} yields immediately the equation
\[0=1-q-p-pq+pq^2+pq^3\]
hence
\[0=1-q-p(1+q-q^2-q^3)
=(1-q)(1-p(1+q)^2),\]
implying, since $q\ge0$, that either $q=1$ or $q=\frac1{\sqrt p}-1$. Since $q\le1$, we conclude that $q=1$ when $p\le1/4$. 
\end{proof}

\newcommand{\cev}[1]{\reflectbox{\ensuremath{\vec{\reflectbox{\ensuremath{#1}}}}}}

\begin{proof}[Proof of Lemma 2]
Let us denote by $y_0$ the location of the first particle that reaches 0, if any, and by $y_1$ the location of the particle that annihilates with the first particle $\bullet_1$, if any. 

For any configuration $\omega$ of particle locations and speeds in $\{\vec\bullet_1\to\bullet\}$, denote by ${\rm rev}(\omega)$ the configuration obtained by reversing the interval $[x_1,y_1]$, that is, the configuration where particles outside $[x_1,y_1]$ are those of $\omega$, and particles inside $[x_1,y_1]$ are symmetric to those of $\omega$ with respect to $\frac{x_1+y_1}2$ and with opposite speeds. 

For $\omega$ in the event defining $r$, that is to say $\omega\in\{0\leftarrow\bullet\}\cap\{\vec\bullet_1\rightarrow\dot\bullet\}$, we clearly have ${\rm rev}(\omega)\in\{0\leftarrow\bullet\}\cap\{\dot\bullet_1\leftarrow\bullet\}$, and notice also that in this case the first bullet reaches $y_1$ before the particle initially at $y_0$ does, i.e.\ $y_1-x_1<y_0-y_1$, and this also holds for ${\rm rev}(\omega)$. Since conversely, for $\omega\in \{0\leftarrow\bullet\}\cap\{\dot\bullet_1\leftarrow\bullet\}\cap\{y_1-x_1<y_0-y_1\}$, we have ${\rm rev}(\omega)\in\{0\leftarrow\bullet\}\cap\{\vec\bullet_1\rightarrow\dot\bullet\}$, we conclude that ${\rm rev}$ is a bijection between $\{0\leftarrow\bullet\}\cap\{\vec\bullet_1\rightarrow\dot\bullet\}$ and $\{0\leftarrow\bullet\}\cap\{\dot\bullet_1\leftarrow\bullet\}\cap\{y_1-x_1<y_0-y_1\}$. Because ${\rm rev}$ preserves the measure, it follows that
\[\P\big((0\leftarrow\bullet)\wedge(\vec\bullet_1\rightarrow\dot\bullet)\big)=\P\big((0\leftarrow\bullet)\wedge(\dot\bullet_1\leftarrow\bullet)\wedge(y_1-x_1<y_0-y_1)\big).\]
We have $ \{0\leftarrow\bullet\}\cap\{\dot\bullet_1\leftarrow\bullet\}=\{\dot\bullet_1\leftarrow\bullet\}\cap\{y_1\leftarrow\bullet\}_{(y_1,\infty)}$, so that, conditional on that event, the distances $y_1-x_1$ and $y_0-y_1$ are independent and have the same distribution, which is atomless. Therefore, 
\[\P\big((0\leftarrow\bullet)\wedge(\dot\bullet_1\leftarrow\bullet)\wedge(y_1-x_1<y_0-y_1)\big)=\frac12\P\big((0\leftarrow\bullet)\wedge(\dot\bullet_1\leftarrow\bullet)\big).\]
To conclude, we finally have
\[\P\big((0\leftarrow\bullet)\wedge(\dot\bullet_1\leftarrow\bullet)\big)=\P\big((\dot\bullet_1\leftarrow\bullet)\wedge(y_1\leftarrow\bullet)_{(y_1,\infty)}\big)=pq^2.\]
\end{proof}

\section{A priori regularity properties}

Let us prove the following result, which in combination with Proposition~\ref{prop:dichotomy} immediately gives Theorem~\ref{thm:main}. 

\begin{proposition}\label{prop:connectivity}
For all $p\in(\frac14,1)$, $\theta(p)>0$. 
\end{proposition}

The proof follows from the two lemmas below. These lemmas respectively rely on two different characterizations of the supercritical phase $\{p\st\theta(p)>0\}$ by means of sequences of conditions about finite subconfigurations. Let us already warn the reader that the definition and properties of the more involved characterization are postponed until the next section. 

\begin{lemma}\label{lem:subcrit_open}
The set of subcritical parameters $\{p\in(\frac14,1)\st\theta(p)=0\}$ is open. 
\end{lemma}

\begin{lemma}\label{lem:supercrit_open}
The set of supercritical parameters $\{p\in(\frac14,1)\st \theta(p)>0\}$ is open.
\end{lemma}

\begin{proof}[{Proof of Proposition~\ref{prop:connectivity}}]
As a conclusion of the above lemmas, the set $A=\{p\in(\frac14,1)\st\theta(p)=0\}$ is both open and closed in $(\frac14,1)$. By connectivity of this interval, it follows that either $A=(\frac14,1)$ or $A=\emptyset$. Since we already know (cf.~\cite{sidoravicius-tournier}) that $A\subset(\frac14,\frac13)$, we deduce that $A=\emptyset$. 
\end{proof}

\begin{proof}[{Proof of Lemma~\ref{lem:subcrit_open}}]
We have $q=\limup_k q_k$ where, for all $k\in\N$, 
\[q_k=\P((0\leftarrow\bullet)_{[0,x_k]}),\]
which gives, using Proposition~\ref{prop:dichotomy},
\begin{align*}
\{p\in(\frac14,1)\st\theta(p)=0\}
	& =\{p\in(\frac14,1)\st q=1\}\\
	& =\{p\in(\frac14,1)\st q>\frac1{\sqrt p}-1\}=\bigcup_{k\in\N}\{p\in(\frac14,1)\st q_k>\frac1{\sqrt p}-1\},
\end{align*}
and each $q_k$ depends only on a configuration of $k$ particles, hence by conditioning on the speeds of these particles we see that $q_k$ is a polynomial in $p$ and therefore is continuous. The lemma follows.
\end{proof}

\begin{proof}[{Proof of Lemma~\ref{lem:supercrit_open}}]
Using the notation $N_k$ from the next section, the upcoming Proposition~\ref{prop:characterization} gives
\[\{p\in(\frac14,1)\st\theta(p)>0\}=\bigcup_{k\in\N}\Big\{p\in(\frac14,1)\st\E[N_k]>0\Big\}, \]
so that the lemma follows by noticing that, as can be seen by conditioning on the speeds of the $k$ particles, the function $p\mapsto\E[N_k]$ is polynomial hence continuous.
\end{proof}

\section{Characterization of the supercritical phase}

While Lemma~\ref{lem:subcrit_open} relies on the simple monotone approximation $q=\limup_k q_k$, where for all $k\in\N$ the probabilities $q_k=\P((0\leftarrow\bullet)_{[0,x_k]})$ depend only on a configuration of $k$ particles, Lemma~\ref{lem:supercrit_open} relies on a formally similar but more involved characterization. This characterization is already alluded to in the first of the final remarks of~\cite{sidoravicius-tournier} as a way to numerically upper bound $p_c$. Given its importance in the present proof, we give it here a more thorough presentation, and show it is necessary and sufficient. 

For all $k\in\N$, consider a random configuration containing only the $k$ bullets $\bullet_1,\ldots,\bullet_k$ (initially located at $x_1,\ldots,x_k$), and denote by $N_k$ the difference between the number of surviving stationary particles and the number of surviving left-going particles: letting $I_k=[x_1,x_k]$, 
\[N_k\defeq\sum_{i=1}^k(\indic_{\dot\bullet_i}-\indic_{\cev\bullet_i})\indic_{(\bullet\not\rightarrow \bullet_i)_{I_k}\wedge(\bullet_i\not\leftarrow \bullet)_{I_k}}\]
In the following, the event in the last indicator function will be written ``$(\bullet_i\text{ survives})_{I_k}$''. 

\begin{proposition}\label{prop:characterization}
For all $p\in(0,1)$, $\theta(p)>0$ $\Leftrightarrow$ $\exists k\ge1,\ \E[N_k]>0$. 
\end{proposition}

\paragraph{\bf Remark.} The fact that $\E[N_1]=\frac12(3p-1)$ recovers (cf.~\cite{sidoravicius-tournier}) that $\theta(p)>0$ when $p>1/3$. The proof of this fact in~\cite{sidoravicius-tournier} is in fact the scheme for the general one given below. Considering $\E[N_2]$ gives the same condition, however $\E[N_3]=3p^3+7p^2\pb-\frac32p\pb^2-8\pb^3$ (where $\pb=\frac{1-p}2$) yields the value $0.32803$ from the remark in~\cite{sidoravicius-tournier}. As the proposition shows, pushing this method further would give arbitrarily good numerical approximations of $p_c$. Let us remind that, although such approximations are rendered pointless by Theorem~\ref{thm:main}, the \textit{existence} of this method still is a theoretical tool in the proof of the said theorem. 

\begin{proof}
\noindent{\it Direct implication.} 
Assume that $\theta(p)>0$. Let us decompose $N_k=\dot N_k-\cev N_k$, where $\dot N_k$ and $\cev N_k$ respectively denote the number of stationary and left-going particles among $\bullet_1,\ldots,\bullet_k$ that survive in restriction to $[x_1,x_k]$. 

For any integer $i$, the event $\{\dot\bullet_i\text{ survives}\}_I$ decreases with the interval $I$ (containing $x_i$). If indeed $\bullet_i$ is stationary and is annihilated by a bullet inside an interval $I$, then introducing new bullets outside $I$ can possibly change the side from which $\bullet_i$ is hit, but not the fact that this bullet is hit. In particular, the number of stationary bullets among $\bullet_1,\ldots,\bullet_k$ that survive in restriction to $[x_1,x_k]$ is larger than or equal to the number of such bullets that survive in ``restriction'' to the whole real line. Taking expectations, by translation invariance of the process on $\R$ this gives
\[\E[\dot N_k]\ge k \P\big((\dot\bullet_1\text{ survives})_\R\big)=kp\theta(p),\]
hence in particular $\E[\dot N_k]\to+\infty$ as $k\to\infty$. 

On the other hand, $\E[\cev N_k]$ is uniformly bounded in $k$. Indeed, $\cev N_k$ clearly grows with $k$, and its limit $\cev N_\infty=\limup_k \cev N_k$ is the number of surviving left-going particles in $(0,\infty)$, and this number has geometric distribution with parameter $1-q>0$ (notice indeed that the configuration on the right of a surviving left-going particle is identically distributed as the configuration on $(0,\infty)$, up to translation) and therefore is integrable. 

We conclude that $\E[N_k]=\E[\dot N_k]-\E[\cev N_k]\ge kp\theta(p)-\frac q{1-q}\to +\infty$ as $k\to\infty$, hence $\E[N_k]>0$ for large~$k$. 

\noindent{\it Reverse implication.}
Assume now that $\E[N_k]>0$ for some $k\ge1$. 

For positive integers $i<j$, define $N(i,j)$ in the same way as $N_k$ except that only the bullets $\bullet_i,\ldots,\bullet_j$ are considered instead of $\bullet_1,\ldots,\bullet_k$. With this notation, $N_k=N(1,k)$. This function $N$ satisfies ``almost'' a superadditivity property. 

\begin{lemma}\label{lem:superadditivity}
Let $k<l$ be positive integers. For any configuration $\omega$ which, in restriction to $[x_1,x_k]$, has no surviving right-going particle, we have
\[N(1,l)\ge N(1,k)+N(k+1,l).\]
\end{lemma}

\begin{proof}[{Proof of Lemma~\ref{lem:superadditivity}}]
When the configurations in $I=[x_1,x_k]$ and in $J=[x_{k+1},x_l]$ are combined, the surviving left-going particles from $J$ can interact with particles from $I$. Each of them either annihilates with a surviving stationary particle (hence giving the same $0$ contribution to both hand sides) or annihilates with a stationary particle that was annihilated in restriction to $I$ hence unleashes its right-going peer which can either survive (making the left-hand side greater by 1), annihilate with a surviving left-going particle (making the left-hand side greater by 2), annihilate with a surviving stationary particle (keeping sides equal) or again annihilate with a stationary particle that was annihilated in restriction to $J$ hence unleash its left-going peer which is offered the same range of possibilities as the particle we first considered. Thus in any case the identity remains satisfied after the effect of each of these left-going particles is taken into account. 
\end{proof}

We shall progressively explore the configuration, starting from 0 and going to the right, by repeating the following two steps: first, discover the next $k$ particles, and then discover the least necessary number of particles until there is no surviving right-going particle in the whole discovered region. We will denote by $K_0=0, K_1, K_2,\ldots$, the number of particles discovered in total after each iteration, and by $\Nt^{(1)}(=N_k),\Nt^{(2)},\ldots$ the quantity computed analogously to $N_k$ but on the newly discovered block of $k$ particles at each iteration, i.e., for all $n$, $\Nt^{(n+1)}=N(K_n+1,K_n+k)$. Let us explain the first iteration in some more detail.

We start by considering the first $k$ particles. Let $\Nt^{(1)}=N(1,k)$. If, in the configuration restricted to $[x_1,x_k]$, no right-going particle survives, then we let $K_1=k$. Otherwise, let $\tau_0$ denote the index of the leftmost surviving right-going particle, and appeal for instance to~\cite[Lemma 3.3]{sidoravicius-tournier} to justify the existence of a minimal $\gamma_1$ such that the event $\{\vec\bullet_{\tau_0}\to\bullet_{\gamma_1}\}_{[\tau_0,\gamma_1]}$ happens, and let $K_1=\gamma_1$. By definition we have that, in both cases, in restriction to $[x_1,x_{K_1}]$, there is no surviving right-going particle and $\Nt^{(1)}=N(1,K_1)$. We then keep iterating this construction: define $\Nt^{(2)}=N(K_1+1,K_1+k)$, and keep exploring on the right of $\bullet_{K_1+k}$ until no surviving right-going particle remains, define $K_2$ to be the index that was reached, and so on. By this construction, the random variables $\Nt^{(n)}$ are i.i.d.\ with same distribution as $N_k$, and for all $n$ we have $N(1,K_n+k)=N(1,K_{n+1})$ and there is no surviving right-going particle in restriction to $[x_1,x_{K_{n+1}}]$. Thus, by repeatedly using the lemma, we have for all $n$, 
\[N(1,K_n)\ge \Nt^{(1)}+\cdots+\Nt^{(n)}.\]
However, by the assumption and the law of large numbers, with positive probability $\Nt^{(2)}+\cdots+\Nt^{(n)}>0$ for all $n\ge2$. Therefore, still with positive probability, it may be that the first $k$ particles are stationary (hence $\Nt^{(1)}=k$) and that $\Nt^{(1)}+\cdots+\Nt^{(n)}> k$ for all $n\ge2$, so that $N(1,K_n)> k$ for all $n\ge2$. This event ensures that 0 is never hit: indeed after the $n$-th iteration of the exploration (for $n\ge2$) there are at least $k+1$ surviving stationary particles due to the definition of the event, but at most $k$ of them can be annihilated by the particles discovered between $K_n$ and $K_{n+1}$, hence by induction the first stationary particle survives forever and prevents 0 from being hit. Thus $\theta(p)>0$.
\end{proof}

\paragraph{\bf Remark}
In the discrete ballistic annihilation model introduced in~\cite{junge2}, the analog of Lemma 2 is wrong due to triple collisions. The same arguments indeed give
\[\hat r=\P_\Z(D>D')p\hat q^2<\frac12 p\hat q^2,\]
where $D$ is the location of the first particle that reaches zero, and $D'$ is an independent copy of $D$. Since $D$ is integer valued, it holds more precisely that
\[\P_\Z(D>D')=\frac12\P_\Z(D\ne D')=\frac12\big(1-\P_\Z(D=D')\big)\]
and $\P_\Z(D=D')$ can be interpreted as the probability that, on the full line, a given stationary particle is involved in a triple collision. 
From $\hat r<\frac12p\hat q^2$, the computation done in the proof of Theorem~\ref{thm:main} shows that, if $\hat q<1$, then $0<1-p(1+\hat q)^2$, hence $\hat q<\frac1{\sqrt p}-1$ and thus the surviving probability of a stationary particle on the full line satisfies
\[\psi(p)= (1-\hat q)^2>\Big(2-\frac1{\sqrt p}\Big)^2=\theta(p).\]
Thanks to this dichotomy, we can argue as in the original model that, for all $p>\frac14$, $\psi(p)>0$ hence furthermore $\psi(p)>\theta(p)$. This comparison was heuristically expected in~\cite{junge2}.

\bibliographystyle{acm}
\bibliography{biblio}

\end{document}